\documentclass[12pt]{amsart}

\usepackage[utf8]{inputenc}
\usepackage[english]{babel}
\usepackage[T1]{fontenc}
\usepackage{mathrsfs,amssymb,amsthm,amsmath,amsfonts}
\usepackage{appendix}

\theoremstyle{plain}
\newtheorem{theorem}{Theorem}
\theoremstyle{definition}

\theoremstyle{plain}
\newtheorem{lemma}{Lemma}
\theoremstyle{plain}

\theoremstyle{plain}

\theoremstyle{plain}
\newtheorem{proposition}[theorem]{Proposition}
\theoremstyle{definition}

\title[Remarks on expansions of the real field]{Remarks on expansions of the real field: tameness, Hardy fields and smooth rings}

\author[R. F.]{Rodrigo Figueiredo}
\address{Departamento de Matem\'atica, Instituto de Matem\'atica e Estat\'is\-tica, Universidade de S\~ao Paulo, Rua do Mat\~ao 1010, CEP 05508-090, S\~ao Paulo, SP, BRAZIL.}
\email{rodrigof@ime.usp.br}

\author [H. L. M.]{Hugo Luiz Mariano}
\address{Departamento de Matem\'atica, Instituto de Matem\'atica e Estat\'is\-tica, Universidade de S\~ao Paulo, Rua do Mat\~ao 1010, CEP 05508-090, S\~ao Paulo, SP, BRAZIL.}
\email{hugomar@ime.usp.br}

\date{February 2019}

\begin{document}

\maketitle

\begin{abstract}
In the talk	\cite{vandendries-matthias-} presented at the Logic and Foundations section of ICM-2018, Rio de Janeiro, the authors analyze, under a model-theoretic perspective, three ways  to enrich the real continuum by infinitesimal and infinite quantities.
In the present work, we establish a first model-theoretic connection of another (but related to the previous one) triple of structures: o-minimal structures, Hardy fields and smooth rings.
\end{abstract}



\section*{Introduction}

According Matthias Aschenbrenner, Lou van den Dries and Joris van der Hoeven (\cite{vandendries-matthias-}):

{{\em ``Germs of real-valued functions, surreal numbers, and transseries are three ways
to enrich the real continuum by infinitesimal and infinite quantities. Each of these
comes with naturally interacting notions of ordering and derivative''.}}

They examine this tripod by the model-theoretic analysis of the category of H-fields, which provides a common framework for these structures.

In the present (short) work, we give some model-theoretic connections among the elements of another tripod, although related to the previous one: o-minimal structures, Hardy fields and smooth rings. 

{\bf Overview of the paper}. In the first section we present the preliminary definitions and  results on o-minimal structures, Hardy fields and $\mathcal{C}^{\infty}$-rings needed in the sequel. Section 2 presents the connections among the elements of the concerned tripod. A final section is devoted to the sketch of possible future works around this subject. In order to write a (brief though) reader-friendly text, we also include an appendix containing some basic results on extensions of real  smooth functions.

\section{Preliminaries}

For the reader's convenience we provide below a simplified account on the three subjects pointed out in the title of this paper. 

\subsection{O-minimality and definability}

It is well-known that the theory of algebraically closed fields is \textit{strongly minimal}, i.e. the definable unary subsets of an algebraically closed field are finite or cofinite: this is a direct consequence of the elimination of quantifiers. 

Analogously, since the theory of real closed fields in the language of the ordered rings admits quantifier elimination, it is an \textit{o-minimal theory}, i.e. the first-order definable unary subsets of a real closed field are finite unions of points and open intervals. The corresponding topology generated over finite cartesian products of definable sets is well-behaved or ``tame''. For a full treatment of o-minimal structures from a geometric viewpoint, see \cite{vandendries-tame} and \cite{vandendries-miller1996}.

Some variants of the notion of o-minimality have been studied since its systematization in the middle of the 1980s. We recall from \cite{fornasiero-servi}, for instance, that a sequence $\mathcal{S}\mathrel{\mathop:}=(\mathcal{S}_n)_{n\geq 1}$, where each $\mathcal{S}_n$ is a collection of subsets of $\mathbb{R}^n$, is called a \textit{weak structure} over the real field if, for all $m,n\geq 1$, the following conditions are satisfied:
\begin{enumerate}
    \item[(WS1)] if $A,B\in \mathcal{S}_m$, then $A\cap B\in \mathcal{S}_m$;
    \item[(WS2)] $\mathcal{S}_m$ contains all zero-sets of polynomials in $\mathbb{R}[X_1,\ldots,X_m]$;
    \item[(WS3)] if $A\in \mathcal{S}_m$ and $B\in \mathcal{S}_n$, then $A\times B\in \mathcal{S}_{m+n}$;
    \item[(WS4)] $\mathcal{S}_m$ is closed under permutation of the variables.
\end{enumerate}
If, in addition, the elements in $\mathcal{S}_1$ are just finite unions of connected components of $\mathbb{R}$, then $\mathcal{S}$ is said to be an \textit{o-minimal} weak structure over the real field.

\subsection{Hardy fields}

Let $f : X \to \mathbb{R}$ and $g : Y \to \mathbb{R}$ be continuous real functions,  where the subsets $X, Y \subseteq \mathbb{R}$  contain an open interval of the form $(c, +\infty)$. These functions are said to have the same \textit{germ at the infinity} (shortly, \textit{germ at $+\infty$}), denoted $[f]=[g]$, if they agree on some open interval $(a,+\infty) \subseteq X \cap Y$. Clearly, this determines an equivalence relation on the set of such functions. The formed quotient set $Q$ supports a natural structure of commutative unitary ring of characteristic zero, with pointwise defined addition and multiplication of germs at infinity. Moreover, this structure can be enriched by two binary relations: $\leq$, the pointwise defined partial ordering; and $\preceq$, the preorder of dominance described as follows. We write $[f] \preceq [g]$ and say that $[f]$ \textit{dominates} $[g]$ iff there exists $b >0$ such that $|f(x)| \leq b|g(x)|$ \textit{eventually}, i.e., there exists $d \in \mathbb{R}$ such that $(d,+\infty) \subseteq X \cap Y$ and $|f(x)| \leq b|g(x)|$ holds for each $x > d$. Two such continuous functions $f, g$ have the same \textit{order of growth at the infinity} when $[f] \preceq [g]$ and $[g] \preceq [f]$.

A subfield $F$ of the ring $Q$ is a {\em Hardy field} when it is closed under differentiation, i.e. if $[f] \in F$, then $[f'] \in F$. There is an interesting class of  first-order structures, the class of {\em H-fields}, formed by ordered differential fields satisfying some further conditions, which includes all Hardy fields expanding $\mathbb{R}$.

There is a strong relationship between Hardy fields and o-minimal structures brought by Chris Miller, for instance, in \cite{miller} (see Proposition \ref{4} below). A striking result (\cite{miller0}) afforded by the combination of these two kinds of structures is the dichotomy, also obtained by Chris Miller, for o-minimal expansions of the real field: either they are polinomially bounded or define the exponential function.

We finish this subsection by asserting the following technical result on extensions of smooth functions defined on open subsets of the real line, which is an easy consequence of the smooth version of Tietze extension theorem (see Appendix A for more details).

\begin{proposition}\label{3}
For any $\mathcal{C}^\infty$ function $g\colon (b,+\infty)\to \mathbb{R}$ with $b\in \mathbb{R}$ and for each $c > b$, there exists a $\mathcal{C}^\infty$ function $\widetilde{g}\colon \mathbb{R}\to \mathbb{R}$ such that $\widetilde{g}=g$ on $(c,+\infty)$.  
\end{proposition}

\subsection[$\mathcal{C}^\infty$-rings]{$\boldsymbol{\mathcal{C}^\infty}$-rings}

Roughly speaking, a $\mathcal{C}^{\infty}$-ring is an $\mathbb{R}$-algebra satisfying additional conditions. The original motivation to introduce and study $\mathcal{C}^{\infty}$-rings was to construct topos-models for Synthetic Differential Geometry (see \cite{mr1}).

Precisely, a \textit{$\mathcal{C}^\infty$-ring} (or \textit{smooth ring}) is a set $A$ together with operations $\Phi_f\colon A^m\to A$ for all $m\geq 0$ and smooth functions $f\colon \mathbb{R}^m\to \mathbb{R}$, where by convention $A^0$ is the single point $\{\emptyset\}$. These operations must satisfy the conditions:
if $f_1,\ldots,f_n\colon\mathbb{R}^m\to \mathbb{R}$ and $g\colon \mathbb{R}^n\to \mathbb{R}$ are smooth functions, and $h\colon \mathbb{R}^m\to \mathbb{R}^n$ is given by 
$$
h(x_1,\ldots,x_m)\mathrel{\mathop:}=g(f_1(x_1,\ldots,x_m),\ldots,f_n(x_1,\ldots,x_m))
$$
then 
$$
\Phi_h(c_1,\ldots,c_m)=\Phi_g(\Phi_{f_1}(c_1,\ldots,c_m),\ldots,\Phi_{f_n}(c_1,\ldots,c_m)),
$$ 
for all $c_1,\ldots,c_m\in A$; for all $1\leq j\leq m$, $\Phi_{\pi_j}=\Pi_j$, where $\pi_j\colon \mathbb{R}^m\to \mathbb{R}$ and $\Pi_j\colon A^m\to A$ denote the projections onto the $j$th terms of $m$-tuples (see \cite{dj} for more details). 

In particular, since each real polynomial function is smooth, then every $\mathcal{C}^\infty$-ring is an $\mathbb{R}$-algebra. Since the theory of $\mathcal{C}^\infty$-rings is equational, the corresponding category admits many interesting constructions, particularly it has all (small) limits and colimits and each set $X$ freely generates a $\mathcal{C}^\infty$-ring, namely $F(X)\mathrel{\mathop:}= \text{colim}_{Y\subseteq_\text{fin} X} \mathcal{C}^\infty(\mathbb{R}^{Y}, \mathbb{R})$.

Every (non trivial) $\mathcal{C}^\infty$-ring $A$ is semi-real, in fact $ 1 + \sum A^2 \subseteq A^\times$.  In Theorem 2.10 in \cite{mr}, it is established that any \textit{$\mathcal{C}^\infty$-field} -- i.e. a $\mathcal{C}^\infty$-ring such that its underlying ring is a field -- is real closed\footnote{In fact, in Theorem 2.10' in \cite{mr} it is shown that every $\mathcal{C}^\infty$-field satisfies an even stronger condition: they are \textit{$\mathcal{C}^\infty$}-real closed.}. This suggests the search for connections between the areas of $\mathcal{C}^\infty$-rings and o-minimal structures.

\section{A first connection among o-mininal structures, Hardy fields and smooth rings}

For each $n\in \mathbb{N}$, let $\mathcal{C}^\infty(\mathbb{R}^n)$ denote the set of all smooth functions from $\mathbb{R}^n$ to $\mathbb{R}$, which is a commutative ring with unity when equipped with the usual pointwise operations, and let $\mathcal{F}=(\mathcal{F}_n)_{n\in \mathbb{N}}$ be a sequence with $\mathcal{F}_n\subseteq \mathcal{C}^\infty(\mathbb{R}^n)$.

Throughout this section, $\mathcal{A}$ designates the expansion of the ordered real field $(\mathbb{R},<,+,\cdot,0,1)$ by the set $\bigcup_{n\in \mathbb{N}}\mathcal{F}_n$.  Unless otherwise stated, by ``definable'' we mean ``definable in $\mathcal{A}$ with parameters from $\mathbb{R}$''.

\begin{proposition}[Proposition 3.1, \cite{miller}]\label{4}
If $\mathcal{R}$ is an expansion of the real field $\mathbb{R}$, then the following are equivalent:
\begin{enumerate}
    \item[(1)] $\mathcal{R}$ is o-minimal;
    \item[(2)] the germs at $+\infty$ of definable in $\mathcal{R}$ unary functions form a Hardy field;
    \item[(3)] every unary definable in $\mathcal{R}$ function is either eventually zero or eventually nonzero.
\end{enumerate}
\end{proposition}

\begin{theorem}\label{5}
If $\mathcal{A}$ is o-minimal, then the commutative ring (with unity) $\mathcal{H}_\mathcal{A}$ of germs at $+\infty$ of definable $\mathcal{C}^\infty$ unary functions is a Hardy field, and is isomorphic to a subfield of a $\mathcal{C}^\infty$-ring. 
\end{theorem}
\begin{proof}
Recall that the commutative ring with unity $\mathcal{H}$ of the germs at $+\infty$ of definable unary functions is a Hardy field (Proposition \ref{4}). For the first part of the theorem, it thus suffices to show that $\mathcal{H}_{\mathcal{A}}$ is a subfield of $\mathcal{H}$ (for which it is sufficient to guarantee that each nonzero element in $\mathcal{H}_{\mathcal{A}}$ is a unit), and $\mathcal{H}_{\mathcal{A}}$ is closed under differentiation. Indeed, any non eventually zero definable $\mathcal{C}^\infty$ function $f$ is eventually nonzero by virtue of Proposition \ref{4}. Hence, there exists $c\in \mathbb{R}$ such that $f$ does not vanish on $(c,+\infty)$. Let $g\colon (c,+\infty)\to \mathbb{R}$ be the function given by $g\mathrel{\mathop:}=1/f$. Clearly, $g$ is definable and $\mathcal{C}^\infty$ (therefore $[g]\in \mathcal{H}_{\mathcal{A}}$), and the equality $[f][g]=[1]$ holds. Now, if $f\colon (c,+\infty)\to \mathbb{R}$ is definable $\mathcal{C}^\infty$ for some $c\in \mathbb{R}$, then $f'\colon (c,+\infty)\to \mathbb{R}$ defined as $x\mapsto f'(x)\mathrel{\mathop:}=\lim_{t\to 0}(f(x+t)-f(x))t^{-1}$ is definable and $\mathcal{C}^\infty$ as well, thus $[f']\in \mathcal{H}_{\mathcal{A}}$.

With regard to the second part, let $I$ denote the set of all $\mathcal{C}^\infty$ functions from $\mathbb{R}$ to $\mathbb{R}$ which are eventually zero. Equipped with the operations induced by those on $\mathcal{C}^\infty(\mathbb{R})$, $I$ is an ideal of the ring $\mathcal{C}^\infty(\mathbb{R})$. (Indeed, observe that the identically zero function is eventually zero and $\mathcal{C}^\infty$; and, the sum of two eventually zero $\mathcal{C}^\infty$ functions is an eventually zero $\mathcal{C}^\infty$ function as well as the multiplication of an eventually $\mathcal{C}^\infty$ function by a $\mathcal{C}^\infty$ function.) We may endow the quotient set $\mathcal{C}^\infty(\mathbb{R})/I$ with a $\mathcal{C}^\infty$-ring structure as follows. For each $f\in \mathcal{C}^\infty(\mathbb{R}^n)$, let $\Phi_f\colon (\mathcal{C}^\infty(\mathbb{R})/I)^n\to \mathcal{C}^\infty(\mathbb{R})/I$ be the map defined as 
$$
\Phi_f(c_1+I,\ldots,c_n+I)\mathrel{\mathop:}=f(c_1,\ldots. c_n)+I,\ c_i\in \mathcal{C}^\infty(\mathbb{R}).
$$ 
To see that $\Phi_f$ does not depend on the representatives $c_1,\ldots,c_n$, consider $c'_1,\ldots,c'_n\in \mathcal{C}^\infty(\mathbb{R})$ so that $c_i+I=c'_i+I$ for each $i=1,\ldots,n$. Hadamard's lemma ensures the existence of $\mathcal{C}^\infty$ functions $g_i\colon \mathbb{R}^{2n}\to \mathbb{R}$ ($i=1,\ldots,n$) with 
$$
f(y_1,\ldots,y_n)-f(x_1,\ldots,x_n)=\sum_{i=1}^n(y_i-x_i)g_i(x_1,\ldots,x_n,y_1,\ldots,y_n),
$$
for all $x_i,y_i\in \mathbb{R}$. The equality of functions
$$
f(c'_1,\ldots,c'_n)-f(c_1,\ldots,c_n) =\sum_{i=1}^n(c'_i-c_i)g_i(c'_1,\ldots,c'_n,c_1,\ldots,c_n)
$$
thus follows. Because $c'_i-c_i\in I$ and $I$ is an ideal, the right-hand side of the above equality lies in $I$. So, the values of $\Phi_f$ at the tuples $(c_1+I,\ldots,c_n+I)$ and $(c'_1+I,\ldots,c'_n+I)$ are the same. This concludes the well definition of $\Phi_f$. It is not hard to see that the sequence $(\Phi_f)_f$ satisfies the defining conditions of a $\mathcal{C}^\infty$-ring, in particular, $(\mathcal{C}^\infty(\mathbb{R})/I,(\Phi_f)_f)$ is a commutative ring with unity. 

Now, we take $T\colon \mathcal{H}_{\mathcal{A}}\to \mathcal{C}^\infty(\mathbb{R})/I$ to be the map given by the rule
$$
[f]\mapsto g+I,
$$
where $g\colon \mathbb{R}\to \mathbb{R}$ is a $\mathcal{C}^\infty$ function (not necessarily definable) with $g\in [f]$. Let $f_1$ and $f_2$ be definable $\mathcal{C}^\infty$ unary functions with $[f_1]=[f_2]$. By Proposition \ref{3}, there exist total $\mathcal{C}^\infty$ unary functions $g_1$ and $g_2$ and real numbers $c_1$ and $c_2$ satisfying $g_1|_{(c_1,+\infty)}=f_1$ and $g_2|_{(c_2,+\infty)}=f_2$. Since $f_1$ and $f_2$ are eventually the same, $g_1-g_2\in \mathcal{C}^\infty(\mathbb{R})$ is eventually zero. Therefore, $g_1-g_2$ lies in $I$. This shows that $T$ is well defined. 

In order to prove that $T$ is an injective ring with unity homomorphism, let $[f_1]$ and $[f_2]$ be germs in $\mathcal{H}_{\mathcal{A}}$ and consider functions $g_1$, $g_2$ and $g$ in $\mathcal{C}^\infty(\mathbb{R})$ with $g_1\in [f_1]$, $g_2\in [f_2]$ and $g\in [f_1+f_2]$. In view of the definition of germ at $+\infty$, there exists a real number $c$ such that $g=f_1+f_2=g_1+g_2$ holds on $(c,+\infty)$, thereby $g-(g_1+g_2)$ belongs to $I$. Therefore, we have $T([f_1]+[f_2])=T([f_1+f_2])=g+I=g_1+g_2+I=T([f_1])+T([f_2])$. Similarly, $T([f_1][f_2])=T([f_1])T([f_2])$. Also, from the construction of $T$, it follows immediately that $T([1])=1+I$, where $1$ denotes, as an abuse of notation, the constant total function $1\in \mathbb{R}$. Finally, if $T([f])=0+I$, where $0$ is the zero function, then by the definition of $T$ we have $[0]=[f]$.

Thus, $T$ is an isomorphism from $\mathcal{H}_{\mathcal{A}}$ onto $\text{Im}(T)$. From the first part of the theorem, it follows that $\text{Im}(T)$ is a subfield of the $\mathcal{C}^\infty$-ring $\mathcal{C}^\infty(\mathbb{R})/I$. 
\end{proof}

In what follows we show that the conclusion of Theorem \ref{5} still holds under weaker assumptions, namely the order minimality is imposed only on the zero-sets of functions definable in $\mathcal{A}$.

It is readily seen that the zero-sets of all definable functions from $\mathbb{R}^n$ to $\mathbb{R}$ ($n\geq 1$) form a weak structure, denoted $\mathcal{Z}=(\mathcal{Z}_n)_{n\geq 1}$.

\begin{theorem}\label{6}
Suppose the weak structure $\mathcal{Z}$ is o-minimal. Then  the commutative ring with unity $\mathcal{H}_\mathcal{A}$ of germs at $+\infty$ of definable $\mathcal{C}^\infty$ unary functions is a Hardy field, and is isomorphic to a subfield of a $\mathcal{C}^\infty$-ring. 
\end{theorem}
\begin{proof}
The proof of this theorem is entirely similar to that of Theorem \ref{5}, except for the assertion that $\mathcal{H}_{\mathcal{A}}$, the set of all germs at $+\infty$ of definable $\mathcal{C}^\infty$ unary functions, is a Hardy field. Such a conclusion is achieved if we show that every nonzero element in $\mathcal{H}_{\mathcal{A}}$ is a unit. For this, let $f\colon (c,+\infty)\to \mathbb{R}$ be a $\mathcal{C}^\infty$ definable unary function, which is not eventually zero. The extension $\widetilde{f}\colon \mathbb{R}\to \mathbb{R}$ of $f$ by zero, that is, $f\mathrel{\mathop:}=f$ on $(c,+\infty)$, and $f\mathrel{\mathop:}=0$ on $(-\infty,c]$, is then definable. (In the case $c=-\infty$, we take $\widetilde{f}$ to be the function $f$ itself.) Consequently, $\widetilde{f}^{-1}(0)\in \mathcal{Z}_1$. Since $\widetilde{f}^{-1}(0)$ is a finite union of connected components of the real line and $\widetilde{f}$ is not eventually zero, none of these connected components is of the form $(d,+\infty)$, with $d\in \mathbb{R}$, nor $\widetilde{f}$ satisfies: for all $x\in \mathbb{R}$ there exists $y>x$ at which $\widetilde{f}$ vanishes. This amounts to the following situation: $f$ is either eventually positive or eventually negative. Therefore, as in the proof of Theorem \ref{5}, it has a multiplicative inverse in $\mathcal{H}_{\mathcal{A}}$. 
\end{proof}

\section{Concluding remarks and future works}

In the paper \cite{vandendries-matthias-}, which has inspired this modest contribution, the authors bring to one's attention a triple of structures that are linked by the notion of H-field  which provides a common framework for these structures. They present a  model-theoretic analysis of the category of
H-fields, e.g. the theory of H-closed fields is model complete,  and relate these results with the original tripod: Hardy fields, surreal numbers and transseries.

In the same vein, we intend to analyze the class of $\mathcal{C}^\infty$-fields under a model-theoretic perspective since we believe that there are clues that this class should satisfy many interesting logical properties:
\begin{itemize}
    \item under real algebra perspective: as already mentioned, every $\mathcal{C}^\infty$-field is ($\mathcal{C}^\infty$)- real closed;
    \item under differential algebra perspective:  since every $\mathcal{C}^\infty$-ring $A$ is isomorphic to quotient of a free $\mathcal{C}^\infty$-ring on some set $X$ by an ideal $I$, $A \cong \mathcal{C}^\infty(\mathbb{R}^X)/I$, it encodes many algebraic derivations.
\end{itemize}

These observations suggest the existence of a relation between the triple in \cite{vandendries-matthias-} and the one here presented that is even stronger than just having a common vertex: Hardy fields.
In particular,  is there a general first-order theory that includes naturally $\mathcal{C}^\infty$-fields and the H-fields?

\appendix
\section{Extending smooth real functions}

We follow closely Chapter 13 and Appendix C in \cite{tu}. 

Recall from elementary calculus that the function $f\colon \mathbb{R}\to \mathbb{R}$, defined as
\[
    f(t)\mathrel{\mathop:}=
\begin{cases}
    e^{-\frac{1}{t}},& \text{if } t> 0\\
    0,              & \text{if } t\leq 0
\end{cases},
\]
is $\mathcal{C}^\infty$. Let $g\colon \mathbb{R}\to \mathbb{R}$ be 
$$
g(t)\mathrel{\mathop:}=\frac{f(t)}{f(t)+f(1-t)}.
$$ 
So $g$ is also $\mathcal{C}^\infty$. For any two positive real numbers $a<b$, the $\mathcal{C}^\infty$ function $\rho\colon \mathbb{R}\to \mathbb{R}$ given by 
$$
\rho(x)\mathrel{\mathop:}=1-g\left(\frac{(x-q)^2-a^2}{b^2-a^2}\right),
$$
where $q\in \mathbb{R}$, is called a \textit{$\mathcal{C}^\infty$ bump function} at $q$. A short description of $\rho$ is made as follows: $\rho$ vanishes on $\mathbb{R}\setminus (q-b,q+b)$, equals $1$ on $[q-a,q+a]$, is strictly increasing on $[q-b,q-a]$, and is strictly decreasing on $[q+a,q+b]$.  


Given a map $f\colon X\to Y$ defined on a topological space $X$, we denote by $\text{supp}\,f$ the set $\overline{\{x\in X\,:\, f(x)\neq 0\}}$ and call it the \textit{support of $f$}.

\begin{lemma}[Partitions of unity]\label{1}
Let $\{U_\alpha\}_{\alpha\in \Lambda}$ be an open cover of $\mathbb{R}$. Then, there is a countable family $\{\varphi_k\}_{k=1}^\infty$ of $\mathcal{C}^\infty$ functions $\varphi_k\colon \mathbb{R}\to \mathbb{R}$ satisfying the conditions 
\begin{enumerate}
    \item[(1)] $\{\emph{supp}\,\varphi_k\}_{k=1}^\infty$ is \emph{locally finite}, that is, any real number has a neighborhood that intersects only finitely many $\emph{supp}\,\varphi_k$;
    \item[(2)] each $\emph{supp}\, \varphi_k$ is compact;
    \item[(3)] for each $k$ there exists an $\alpha\in\Lambda$ with $\emph{supp}\,\varphi_k\subseteq U_\alpha$;
    \item[(4)] $\sum_{k=1}^\infty\varphi_k=1$.
\end{enumerate}
Such a collection $\{\varphi_k\}_k$ is called a \emph{$\mathcal{C}^\infty$ partition of unity subordinate to the cover} $\{U_\alpha\}_{\alpha\in \Lambda}$.
\end{lemma}
\begin{proof}
Our aim is to find, for each integer $m$, finitely many bounded open intervals $W_1^m,\ldots, W_{l(m)}^m$ and finitely many $\mathcal{C}^\infty$ bump functions $\phi_1^m,\ldots,\phi_{l(m)}^m$ such that 
\begin{enumerate}
    \item[(a)] $W_1^m,\ldots,W_{l(m)}^m$ cover $[m,m+1]$;
    \item[(b)] $\phi_j^m>0$ on $W_j^m$, and $\phi_j^m=0$ on $\mathbb{R}\setminus W_j^m$ for $j=1,\ldots,l(m)$;
    \item[(c)] $\text{supp}\,\phi_j^m\subseteq U_{\alpha_{mj}}\cap (m-1,m+2)$ for some $\alpha_{mj}\in \Lambda$;
    \item[(d)] $\text{supp}\,\phi_j^m$ is compact.
\end{enumerate}
Note that the collection $\{\text{supp}\,\phi_j^m\}$ of the supports of the functions $\phi_j^m$ is locally finite, since any $[m,m+1]$ does not intersect $\text{supp}\,\phi_j^n$, for all $n\geq m+3$, $n\leq m-3$, and $j=1,\ldots,l(n)$. Therefore, the sum $\phi\mathrel{\mathop:}=\sum_{m,j}\phi^m_j$ is well defined as a smooth function. Moreover, any real number $x$ lies in some $[m,m+1]$, and by (a) in some open interval $W_j^m$ as well. Hence, $\phi_j^m(x)>0$. This shows that the $\mathcal{C}^\infty$ function $\phi$ is everywhere positive. In view of this, for each integer $m$ and each $j=1,\ldots,l(m)$ the function $\varphi_j^m\colon \mathbb{R}\to \mathbb{R}$, given by 
$$
\varphi_j^m\mathrel{\mathop:}=\frac{\phi_j^m}{\phi},
$$ 
is well defined. 
A routine argument shows that $\{\varphi_j^m\}_{m,j}$ is a family of $\mathcal{C}^\infty$ functions satisfying the conditions (1)-(4). 

It remains to prove the statements (a)-(d). Indeed, for each point $x\in[m,m+1]$, where $m$ is an arbitrary integer, there is an $\alpha\in \Lambda$ with $x\in U_\alpha\cap (m-1,m+2)$, by recalling that $\{U_\alpha\}_\alpha$ covers $\mathbb{R}$. Let $\phi_x$ be a $\mathcal{C}^\infty$ bump function at $x$ which is positive on a bounded open interval $W_x$ centered at $x$ whose closure is included in the open set $U_\alpha\cap (m-1,m+2)$. The set $\text{supp}\,\phi_x$ is then included in $[m-1,m+2]$, therefore $\text{supp}\,\phi_x$ is compact. Since $\{W_x\,:\, x\in [m,m+1]\}$ is an open cover of the compact set $[m,m+1]$, there are finitely many intervals $W_{x_1},\ldots,W_{x_{l(m)}}$, with associated $\mathcal{C}^\infty$ bump functions $\phi_{x_1},\ldots,\phi_{x_{l(m)}}$, which cover $[m,m+1]$. \end{proof}

A function $f\colon A\to \mathbb{R}$, defined on an arbitrary set $A\subseteq \mathbb{R}$, is \textit{of class $\mathcal{C}^\infty$} if for every point $x$ in $A$ there exist an open $U$ containing $x$ and a $\mathcal{C}^\infty$ function $\widehat{f}\colon U\to \mathbb{R}$ such that $\widehat{f}=f$ in $U\cap A$.

\begin{lemma}[Tietze extension theorem]\label{2}
Let $F$ be any closed subset of $\mathbb{R}$ and let $f\colon F\to \mathbb{R}$ be a $\mathcal{C}^\infty$ function. Then there exists a $\mathcal{C}^\infty$ function $\widetilde{f}\colon \mathbb{R}\to \mathbb{R}$ such that $\widetilde{f}|_F=f$.
\end{lemma}
\begin{proof}
By hypothesis, for each $x\in F$ there exist an open subset $U_x$ of the real line and a $\mathcal{C}^\infty$ function $\widehat{f}_x\colon U_x\to \mathbb{R}$ such that $\widehat{f}_x=f$ on $U_x\cap F$. Let $\{\varphi_k\}_{k=1}^\infty$ be a smooth partition of unity subordinate to the open cover $\{U_x\}_{x\in F}\cup \{\mathbb{R}\setminus F\}$ of $\mathbb{R}$. Let us reindex the partition of unity in order to have the same index set as the cover, which gives $\{\varphi_x\}_{x\in F}\cup \{\varphi_0\}$, with $\text{supp}\,\varphi_x\subseteq U_x$ and $\text{supp}\,\varphi_0\subseteq \mathbb{R}\setminus F$. (This can be done by adding to the original family the zero functions.) Now, we extend each $\widehat{f}_x$ to $\mathbb{R}$ by zero, so the function $\varphi_x\widehat{f}_x$ is smooth. Thus we can define $\widetilde{f}\colon \mathbb{R}\to \mathbb{R}$ by 
$$
\widetilde{f}(y)\mathrel{\mathop:}=\sum_{x\in F}\varphi_x(y)\widehat{f}_x(y).
$$
Because the collection $\{\text{supp}\,\varphi_x\}_{x\in F}$ is locally finite, this sum is finite in a neighborhood of every point in $\mathbb{R}$, and therefore defines a $\mathcal{C}^\infty$ function. Note that if $y\in F$ then $\varphi_0(y)=0$, and $f_x(y)=f(y)$ for each $x$ such that $\varphi_x(y)\neq 0$. Then, for any $y\in F$, $\widetilde{f}(y)=\sum_{x\in F}\varphi_x(y)\widetilde{f}_x(y)=\sum_x\varphi_x(y)f(y)=f(y)(\sum_{x\in F}\varphi_x(y)+\varphi_0(y))=f(y)$, \textit{i.e.}, $\widetilde{f}$ is indeed an extension of $f$.
\end{proof}

\noindent\textit{Proof of Proposition \ref{3}}. 
Set $\widehat{g}\mathrel{\mathop:}=g|_{[c,+\infty)}$. Then $\widehat{g}$ is of class $\mathcal{C}^\infty$. Tietze extension theorem thus gives a function $\widetilde{g}\colon \mathbb{R}\to \mathbb{R}$ of class $\mathcal{C}^\infty$ extending $\widehat{g}$ and the proof is finished.



\



\end{document}